\documentclass[12pt,draft,journal,onecolumn]{IEEEtran}%One col TAC

\ifCLASSINFOpdf
\else
\fi

\usepackage{amsmath,color,amsthm}
\usepackage{amssymb}  % assumes amsmath package installed
\usepackage{amsfonts}
\usepackage{graphicx}
\usepackage{dsfont}
\usepackage{psfrag}
\usepackage{subfigure,epstopdf}
\usepackage[norefs,nocites]{refcheck}           % check references (options:showrefs,norefs,
\usepackage{cite}
\usepackage{bbm}

\usepackage{hyperref}
\usepackage{algorithm}
\usepackage{algorithmic}
\usepackage{amsmath}

\newtheorem{theorem}{Theorem}
\newtheorem{corollary}{Corollary}
\newtheorem{remark}{Remark}
\newtheorem{proposition}{Proposition}
\newtheorem{lemma}{Lemma}
\newtheorem{definition}{Definition}

\begin{document}
\title{\huge The Robust Minimal Controllability Problem}

%
%\author[Ramos]{Guilherme Ramos}$^{,\star}$\ead{guilherme.ramos@ist.utl.pt}\qquad     % Add the 
%\author[Ramos,Pequito]{S\'ergio Pequito}$^{,\star}$\ead{spequito@andrew.cmu.edu}    \qquad            % e-mail address 
%\author[Pequito]{Soummya Kar}\ead{soummyak@andrew.cmu.edu}  \hspace{3cm} % (ead) as shown
%\author[Ramos,Aguiar]{A. Pedro Aguiar}\ead{pedro.aguiar@fe.up.pt} \quad   % (ead) as shown
%\author[Ramos2]{Jaime Ramos}\ead{jabr@math.ist.utl.pt}  % (ead) as shown
%
%\address[Ramos]{Institute for System and Robotics, Instituto Superior T\'ecnico,  University of Lisbon, Lisbon, Portugal}  % Please supply                                              
%\address[Pequito]{Department of Electrical and Computer Engineering, Carnegie Mellon University, Pittsburgh, USA}             % full addresses
%\address[Aguiar]{Department of Electrical and Computer Engineering, Faculty of Engineering, University of Porto, Porto, Portugal}        % here.
%\address[Ramos2]{SQIG-Instituto de Telecomunica\c{c}\~oes, Department of Mathematics, Instituto Superior T\'ecnico, \\  University of Lisbon, Lisbon, Portugal}        % here.

%\begin{keyword}                           % Five to ten keywords,  
%Control Systems Design, Controllability, NP-completeness, Structural Controllability. %Linear Systems; Structural Systems Theory; Computational Complexity               % chosen from the IFAC 
%\end{keyword}      

\author{\IEEEauthorblockN{ S\'ergio Pequito $^{\dagger,\ddagger,\ast}$ \quad Guilherme Ramos $^{\ddagger,\ast,\star}$\quad Soummya Kar $^{\dagger}$ \\ A. Pedro Aguiar $^{\ddagger,\diamond}$ \qquad Jaime Ramos $^{\star}$}
\thanks{ $^{\ast}$ The first two authors contributed equally to this work.

$^{\dagger}$ Department of Electrical and Computer Engineering, Carnegie Mellon University, Pittsburgh, PA 15213

$^{\ddagger}$  Institute for Systems and Robotics, Instituto Superior T\'ecnico, Technical University of Lisbon, Lisbon, Portugal

$^{\diamond}$ Department of Electrical and Computer Engineering, Faculty of Engineering, University of Porto, Porto, Portugal

$^{\star}$ SQIG-Instituto de Telecomunica\c{c}\~oes, Department of Mathematics, Instituto Superior T\'ecnico,  University of Lisbon, Lisbon, Portugal

The second author also acknowledge the partial support from the DP-PMI and Funda\c{c}\~ao para a Ci\^encia (FCT), namely through scholarship SFRH/BD/52242/2013
}}

%
%
%\thanks{The first two authors contributed equally to this work. }
%\thanks{This work was partially supported by grant SFRH/BD/33779/2009, from Funda\c{c}\~ao para a Ci\^encia e a Tecnologia  (FCT) and  the CMU-Portugal (ICTI) program, and by projects CONAV/FCT-PT (PTDC/EEACRO/113820/2009), FCT (PEst-OE/EEI/LA0009/2013),  MORPH (EU FP7 No. 288704), and NSF grant 1306128.}
%\thanks{Further, the second author also acknowledge the partial support from the DP-PMI and Fundação para a Ciência e a Tecnologia (Portugal), namely through scholarship SFRH/BD/52242/2013}
%

\maketitle

%\onecolumn

\begin{abstract}
In this paper, we address two minimal controllability problems, where the goal is to determine a minimal subset of state variables in a linear time-invariant system to be actuated to ensure controllability under additional constraints. First, we study the problem of characterizing the sparsest input matrices that assure controllability when the autonomous dynamics' matrix is simple. Secondly, we build upon these results to describe the solutions to the robust minimal controllability problem, where the goal is to determine the sparsest input matrix ensuring controllability when specified number of inputs fail. Both problems are NP-hard, but under the assumption that the dynamics' matrix is simple, we show that it is possible to reduce these two problems to set multi-covering problems. Consequently, these problems share the same computational complexity, i.e., they are NP-complete, but polynomial algorithms to approximate the solutions of a set multi-covering problem can be leveraged to obtain close-to-optimal solutions to either of the minimal controllability problems.
\end{abstract}

\IEEEpeerreviewmaketitle

\vspace{-0.2cm}

\section{Introduction}

The problem of guaranteeing  that a dynamical system can be driven toward the  desired  state regardless of its initial position is a fundamental question that has been  studied in control systems and it is referred to as \emph{controllability}. Several applications, for instance, control processes, control of large flexible structures, systems biology and power systems  \cite{Magnus,largeScale,Skogestad04a} rely on the notion of controllability to safeguard their proper functioning. Furthermore, as the systems become larger (i.e., the dimension of their state space), we (often) aim to  identify a relatively small subset of  state variables that ensure the controllability of the system, for instance, due to economic constraints~\cite{nphard}. Consequently, it is natural to pose the following question. 

$\mathcal Q_1$: \textit{Which state variables need to be directly actuated to ensure the controllability of a dynamical system?}

Question $\mathcal Q_1$ can be formally captured by the  \emph{minimal controllability problem} (MCP)~\cite{nphard} that aims to determine the minimum number of state variables  that need to be actuated to ensure system's controllability.  Unfortunately,  the MCP problem was shown to be NP-hard~\cite{nphard}, which implies that a polynomial solution to determine its solution is unlikely to exist.

The MCP is also fundamental to understand resilience and robustness properties of dynamical systems since it unveils which variable need to be actuated. These resilience/robustness properties are crucial to coping with  the adverse nature of the environments where the actuators are deployed and, due to the wear and tear of the materials, some of these actuators may malfunction over time. In addition, the inputs can malfunction due to  a malicious external agent  who aims to tamper with the inputs to compromise the system behavior. In fact, a classical example  of such malicious  attack  is the Stuxnet malware incident~\cite{langner2011robust}, in which the controller's input response to a tempered measured output lead the system  away from its normal operating conditions.

Therefore, from a design perspective, we would like to deploy actuators in the system such that any subset with at most a specified number of actuators can fail without compromising the controllability of the system. Subsequently, invoking similar reasons to the MCP, we can seek to address the robust MCP (rMCP) that aims to determine the sparsest input matrix that ensures controllability if at most a specified number of actuators fail. It is important to mention that both minimal controllability problems can be stated regarding observability, by invoking the duality between controllability and observability in LTI systems~\cite{Hespanha09}. In particular,~\cite{shoukry2014event,2015arXiv150307125C,2012arXiv1205.5073F} provide necessary and sufficient conditions concerning the sensor deployment to ensure that a reliable estimate of the system is recovered. More importantly, those conditions can be achieved by design, by solving the rMCP.

\textbf{Related Work: } The understanding of which state variables need to be actuated to asseverate certain properties of the system has  been an active research area~\cite{PinningWangS14}. Initially,  the goal was to establish stability and/or asymptotic stability of the dynamics for reference point, for instance, consensus or agreement value~\cite{RenSurvey,rezatac07}. The trend has changed to assure that the system is controllable, since (often)  we want to ensure that a control law exists such that an arbitrary  goal or desired state is achieved in finite time. 

This paper follows up and subsumes some of the existing literature where the dynamics' matrix is assumed to be the Laplacian, symmetric (modeling undirected graphs) and/or irreducible (modeling directed graphs with the digraph representation being a strongly connected component). In~\cite{CirculantNet} the controllability of circulant networks is analyzed by exploring the Popov-Belevitch-Hautus eigenvalue criterion, where the eigenvalues are characterized using the \mbox{Cauchy-Binet} formula. The controllability in multi-agents with Laplacian dynamics was initially explored in~\cite{1428782}. Later, in~\cite{multiagents, EMCCB12}, the controllability for Laplacian dynamics is studied, and necessary and sufficient conditions are given in terms of partitions of the graph.  In~\cite{controlPathCycles}, the controllability is explored for paths and cycles, and later extended by the same authors to the controllability of grid graphs  by means of reductions and symmetries of the graph~\cite{SymmetriesNotarstefano}, and considering dynamics that are scaled Laplacians. In~\cite{RegularDistance,mingcao} the controllability is studied for  strongly regular
graphs and distance-regular graphs. Recently, in~\cite{controlLaplac,chapmanCDC14} new insights on the controllability of  Laplacian dynamics are given regarding the uncontrollable subspace. In addition, in~\cite{pasqZampieri} the controllability of isotropic and anisotropic networks is analyzed.  

Furthermore,~\cite{controlLaplac} concludes by pointing out that further study of non-symmetric dynamics and the controllability is required -- which we address in the present paper. Note that the MCPs lie within the framework of sparse optimization subject to a rank constraint. Further, we notice that the problem addressed does not belong to known classes where polynomial solutions are available~\cite{Foucart}, nor it resources to convex relaxation schemes, where no sub-optimality guarantees are available. Instead, we consider a much less restrictive assumption: $A$ is a \emph{simple} matrix, i.e., all eigenvalues are distinct. Furthermore, there are several applications where $A$ satisfies this assumption, for instance, all dynamical systems modeled as random networks of the Erd\H{o}s-R\'enyi type~\cite{Tao}, as well as several known dynamical systems used as benchmarks in control systems engineering~\cite{OgataControl}. 

Observe that the MCP problem presents both continuous and discrete optimization properties, captured by the controllability property and the number of non-zero entries, respectively. To avoid the nature of this problem, in~\cite{nphard} the non-zero entries of the input matrix were randomly generated. In the present paper, we `decouple' the continuous and discrete optimization properties, and show that by first solving the discrete nature of the problem, it is always possible to deterministically obtain a solution to MCP in a second phase. Besides, the first step reduces the MCP to the set covering problem -- well known to be NP-hard. Nonetheless, the set covering problem is likely one of the most studied NP-hard problems (probably second only to the SAT problem). Subsequently, although the set covering problem is NP-hard, some subclasses of the problem are equipped with sufficient structure that can be leveraged to invoke a polynomial algorithm that approximate the solution with `almost' optimality guarantees~\cite{Bronnimann1995}.  This contrasts with the approach proposed in~\cite{nphard}, where an approximated solution particular to the MCP problem was provided. In addition,  we study the rMCP which has not been previously addressed in the literature. Similarly to the MCP, we show that the rMCP can be polynomially reduced to the \emph{set multi-covering problem}, i.e., a set covering problem that allows the same elements to be covered a predefined number of times. Furthermore, extensions of polynomial approximation algorithms are also available with similar optimality guarantees.

Alternatively, in~\cite{PequitoJournal}  instead of determining the sparsest input matrix  ensuring the controllability, the aim is to determine the sparsest input matrix that ensures \emph{structural controllability}, which we refer to as the \emph{minimal structural controllability problem} (MSCP) -- see Section~\ref{prelimSec} for formal definitions and problem statement. Briefly, the MSCP  focus on the structure of the dynamics, i.e., the location of zeros/non-zeros, and the obtained sparsest input matrix is such that for \emph{almost all} matrices satisfying the structure of the dynamics and the input matrix, the system is controllable~\cite{dionSurvey}. Notwithstanding,   in the present paper, we provide an example where the solution to the minimal structural controllability problem \emph{is not} necessarily a solution to the minimal controllability problem when the dynamics' matrix is simple; hence, \emph{disproving the general belief that a solution to MSCP is a solution to MCP in such cases}.  Further, we emphasize that the solution to the MSCP has been fully explored in \cite{PequitoJournal} and can be determined by recurring to polynomial complexity algorithms; more precisely, $\mathcal O(n^3)$ where $n$ is the dimension of the state space.  In addition, the minimum number of state variables to achieve structural controllability can
account for the scenario where actuating state variables incur in different cost~\cite{PequitoJ3}.
Further, if the collection of possible actuators is given a priori and we seek the minimum
number of these actuators to ensure structural controllability, then the problem is NP-hard~\cite{PequitoCMIS}.
Finally,~\cite{PequitoJ7} studies  the structural counterpart of the rMCP under one failure, which is also proved to be NP-hard, and shown to be reducible to a  \emph{weighted} set covering problem. In particular, the reductions and the objects captured by the sets in the set covering problem in~\cite{PequitoJ7} are entirely different from those of the problems explored in this paper, mainly, due to the nature of the problems. \hfill $\circ$

\textbf{Main Contributions} of the present paper are as follows: (\emph{i}) we characterize the exact solutions to the MCP; (\emph{ii}) we show that for a given dynamics'  matrix almost all input vectors satisfying a specified structure are solutions to the MCP; (\emph{iii}) we prove that the rMCP is an NP-hard problem; (\emph{iv})  we  characterize the exact solutions to the  rMCP; (\emph{v}) we show that the decision version of both MCPs are NP-complete;  (\emph{vi}) we provide approximated solutions to both MCPs and discuss their optimality guarantees;  and, finally, in (\emph{vii}) we discuss the limitations of the proposed methodology. \hfill $\circ$

The remainder of this paper is organized as follows.
In Section~\ref{probFormSec}, we formally state both MCPs addressed in this paper. Next, in Section~\ref{prelimSec}, we review concepts from computational complexity and control systems that are essential to keep this paper self-contained. In Section~\ref{MCPSection}, we present the main results of this paper: we characterize 
the solutions to the MCPs, their complexity, and polynomial algorithms that approximate the solution. Finally, in Section~\ref{secCompStrucvsNonStruct} we provide some examples that illustrate the main results of the paper and discuss the limitations of the proposed methodology.

\vspace{0.4cm}

\textbf{Notation: } We denote vectors by small font letters such as $v,w,b$ and its corresponding entries by subscripts; for example, $v_i$ corresponds to the $i$-th entry in the vector $v$. A collection of vectors is denoted by $\{v^j\}_{j\in \mathcal J}$, where the superscript indicates an enumeration of the vectors using indices from a set (usually denoted by calligraphic letter) such as $\mathcal I,\mathcal J\subset \mathbb{N}$. The  number of elements of a set $\mathcal S$ is denoted by $|\mathcal S|$. Real-valued matrices are denoted by capital  letters, such as $A$, $B$, and $A_{i,j}$ denotes the entry in the $i$-th row and $j$-th column in matrix $A$. We denote by $I_n$ the $n$-dimensional identity matrix. Given a matrix $A$, $\sigma (A)$ denotes the set of eigenvalues of $A$, also known as the \emph{spectrum} of $A$. Given two matrices $M_1\in \mathbb{C}^{n\times m_1}$ and $M_2\in \mathbb{C}^{n\times m_2}$, the matrix $[M_1 \ M_2]$ corresponds to the $n\times (m_1+m_2)$ concatenated complex matrix. 
The structural pattern of a  vector/matrix (i.e., the zero/non-zero pattern) or a \emph{structural vector/matrix} have their  entries in  $\{0,\star\}$, where $\star$ denotes a non-zero entry, and they are denoted by a vector/matrix  with a bar on top of it. In other words, $\bar A$ denotes a matrix with $\bar A_{i,j}=0$ if $A_{i,j}=0$ and $\bar A_{i,j}=\star$ otherwise.  We denote by $A^\intercal$  the transpose of $A$. The function $\cdot : \mathbb{C}^n\times \mathbb{C}^n \rightarrow \mathbb{C}$ denotes the usual inner product in $\mathbb{C}^n$, i.e., $v\cdot w=v^\dagger w$, where $v^\dagger$ denotes the adjoint of $v$ (the conjugate of $v^\intercal$). With some abuse of notation, $\cdot:\{0,\star\}^n\times\{0,\star\}^n\rightarrow \{0,\star\}$ also denotes the map where $\bar v\cdot \bar w\neq 0$, with $\bar v,\bar w\in\{0,\star\}^n$ if and only if there exists $i\in\{1,\ldots,n\}$ such that $\bar v_i=\bar w_i=\star$.  Additionally, $\| v\|_0$ denotes the number of non-zero entries of the vector $v$ in either $\{0,\star\}^n$ or $\mathbb{R}^n$. Given a subspace $\mathcal H\subset\mathbb{C}^n$ we denote by $\mathcal H^\mathsf{c}$ its complement with respect to $\mathbb{C}$, i.e.,  $\mathcal H^\mathsf{c}=\mathbb{C}^n\setminus\mathcal {H}$. In addition, inequalities involving vectors are to be interpreted component-wise. With abuse of notation, we will use inequalities involving structural vectors as well -- for instance, we say $\bar{v}\geq\bar{w}$ for two structural vectors $\bar{v}$ and $\bar{w}$ if and the only if the following two conditions hold: (i) if $\bar w_{i}=0$, then $\bar v_i\in\{0,\star\}$, and  (ii) if $\bar w_{i}=\star$ then $\bar v_i=\star$.

\section{Problems Statement}\label{probFormSec}

In this paper, we focus on dynamical systems modeled by discrete-time linear time-invariant (LTI) systems, but the results are readily applicable to continuous-time LTI systems. We will neglect the output equation because we are only addressing the controllability problem. Therefore, consider systems described by \begin{equation}
x(k+1)=Ax(k)+Bu(k),\quad x(0)=x_0,
\label{dyn}
\end{equation}
where $x \in \mathbb{R}^n$ is the state of the sytem,  $u \in \mathbb{R}^p$ is the input signal exert by the actuators, and $k\in \mathbb{N}$  denotes the time instance. The matrix $A\in\mathbb{R}^{n\times n}$, which is referred to as the system dynamics' matrix describes the coupling between state variables. The matrix $B\in\mathbb{R}^{n\times p}$ is the input matrix and describes the state variables that the inputs act on. As previously mentioned, it is often desirable  the LTI system~\eqref{dyn}  be \emph{controllable}, i.e., a system can be steered towards a desirable state in at most $n$ steps despite the  initial state $x_0$, in which case the pair $(A,B)$ is said to be controllable. 

The first problem addressed in this paper is the MCP, that can be formally stated as follows.
\vspace{0.2cm}

\noindent $\mathcal P_1$:  Given the system dynamics' matrix $A$ determine the input matrix $B\in\mathbb{R}^{n\times n}$ such that
\begin{equation}
\begin{array}{cc}
B^*=\arg\min\limits_{B\in\mathbb{R}^{n\times n}} & \|B\|_0\\
\text{s.t.} & (A,B) \text{ controllable.}
\end{array}
\label{mcpStatement}
\end{equation}
Notice that the input matrix is assumed to be $n\times n$ to ensure a solution to exist, since the identity matrix always ensures system's controllability.

Alternatively, under the adverse scenarios of failure or malicious temper of the actuators, the dynamics of the system can be modeled by
\begin{equation}
x(k+1)=Ax(k)+Bu(k)+a(k),
\label{inputAttackSecurity}
\end{equation}
where the malfunctioning inputs correspond to non-zero entries in $a\in\mathbb{R}^n$ representing an alteration of the actuation in comparison with the actual value. Therefore, an extra set of actuators should be in place to ensure that it is still possible to control the system if some inputs fail, i.e., the system
\begin{equation}
	x(k+1)=Ax(k)+B_{\mathcal M\backslash \mathcal A}u(k),
\label{outputOutAttack}
\end{equation}
is controllable, where $B_{\mathcal M\backslash \mathcal A}$ consists of the subset of columns with indices in $\mathcal M\backslash \mathcal A$, the set  $\mathcal M=\{1,\ldots, p\}$ is the set of inputs' labeling indices and $\mathcal A=\{i\in \mathcal M: \ a_i(k)\neq 0, \ k\in\mathbb{N} \}$  the set of  indices of malfunctioning  actuators.  Therefore, $(A,B_{\mathcal M\backslash \mathcal A})$  is desirable to be controllable, and,  subsequently, the rMCP can be posed as follows.
\vspace{0.2cm}

\noindent $\mathcal P_2$: Given a dynamics' matrix $A\in\mathbb{R}^{n\times n}$ and the number of possible input failures $s$, determine the matrix $B^*\in\mathbb{R}^{n\times (s+1)n}$ such that 
\begin{align}
B^*=\arg \min\limits_{B\in\mathbb{R}^{n\times (s+1)n}}&\quad  \|B\|_0 \label{problemstatementeq0} \\
\text{s.t.} \qquad & (A,B_{\mathcal M\backslash \mathcal A}) \text{ is controllable},\notag\\
&  \ |\mathcal A|\le s,\ \mathcal A\subset \mathcal M,\notag
\end{align}
where $\mathcal M\subset \{1,\dots,n\}$ are the indices of the non-zero columns of the matrix $B$. Notice that, similarly to $\mathcal P_1$, the dimension of $B$ are $n\times (s+1)n$ to ensure that a solution always exist, in particular, in the worst case scenario the matrix  $B$ that concatenates $s$ times the identity matrix is a feasible solution. In practice, only the non-zero columns of $B$ matter, which we refer to as \emph{effective inputs}.

 In this paper, both MCPs proposed above will be addressed under the following two assumptions. 

\textbf{Assumption 1}: The dynamics' matrix is \emph{simple}, i.e., all the eigenvalues of $A$ are distinct. \hfill $\circ$

We notice that Assumption 1 is not very restrictive since there are several applications where $A$ satisfy this assumption. For example,  dynamical systems modeled as random networks of the Erd\H{o}s-R\'enyi type~\cite{Tao}, as well as several known dynamical systems used as benchmarks in control systems engineering~\cite{OgataControl}.

\textbf{Assumption 2}: A  \emph{left-eigenbasis} of $A$ is available, i.e., the eigenbasis consisting of left-eigenvectors of $A$. \hfill $\circ$

The second assumption is required by technical reasons, since an eigenbasis is determined using numerical methods. Therefore, in practice, it may be composed of approximated eigenvectors to a given floating-point error -- see Section~\ref{numericalAnalysis} for further discussion.

\section{Preliminaries and Terminology}\label{prelimSec}

In this section, we review some basic concepts in computational complexity theory, control systems, and structural systems theory, to keep the paper self-contained.

In what follows, we use some concepts of computational complexity theory~\cite{Coo71}, that address the classification of  (computational) problems into complexity classes. 
Formally, this classification is for \emph{decision problems}, i.e., problems with a `yes' or `no' answer.
 Further, for a decision problem, if there exists an algorithm that obtains the correct answer in a number of steps that is bounded by a polynomial  in the size of the input data of the problem, then the algorithm is referred to as an \emph{efficient} or \emph{polynomial} solution to the decision problem and the decision problem is said to be polynomially solvable or belong to the class of polynomially solvable problems. A decision problem  is said to be in NP (i.e., the class of nondeterministic polynomially  solvable problems) if, given any possible solution instance, it can be verified using a polynomial procedure whether the instance constitutes a solution to the problem or not.  It is easy to see that any problem that is polynomially solvable (in P) is also in NP, although, there are some problems in NP for which it is unclear whether polynomial solutions exist. These latter problems are referred to as being NP-complete. Consequently, the class of NP-complete problems contains those that are the \emph{hardest} among the NP problems, i.e., those that are verifiable using polynomial algorithms, but no polynomial algorithms to solve them are known to exist. Whereas the above classification is intended for decision problems, it can be immediately extended to optimization problems, by noticing that every optimization problem can be posed as a decision problem. More precisely, given a minimization problem, we can pose the following decision problem: Is there a solution to the minimization problem that is less than or equal to a prescribed value? On the other hand, if the solution to the optimization problem is obtained, then any decision version can be easily addressed. Consequently, if a (decision) problem is NP-complete, then the associated optimization problem is referred to as being NP-hard. We refer the reader to \cite{Garey:1979:CIG:578533} for an introduction to the topic. In what follows, we will consider the following NP-hard problem.

\begin{definition}[\hspace{-0.01cm}\cite{setMultiCover}]{(Minimum Set Multi-covering Problem)}
\label{setcover}
Given a set of $m$ elements $\,\mathcal{U}=\left\{1,2,\hdots,m\right\}$ referred to as universe, a collection of $n$ sets $\mathcal{S}=\left\{\mathcal S_1,\hdots,\mathcal S_n\right\}$, with $\mathcal S_j\subset \mathcal U$, with $j\in \{1,\hdots,n\}$ $\displaystyle\bigcup_{j=1}^{n}\mathcal S_j=\mathcal{U}$, and a demand function $d:\mathcal U\rightarrow \mathbb{N}$ that indicates the number of times an element $i$ needs to be covered. In other words, $d(i)$ is the minimum number of sets in $\mathcal S$ that need to be consider such that $i$ is member of all of this sets. The  minimum set multi-covering problem consists of finding a set of indices $\,\mathcal{I}^*\subseteq \left\{1,2,\hdots,n\right\}$ corresponding to the minimum number of sets covering $\mathcal U$, where every element $i\in\mathcal U$ is covered at least $d(i)$ times, i.e.,
\[\begin{array}{rl}\mathcal{J}^*=\underset{\mathcal{J}\subseteq \left\{1,2,\hdots,n\right\}}{\arg \min}&\quad |\mathcal{J}|
\\
\text{s.t. } \quad &|\{j \in \mathcal J: i\in \mathcal S_j\}| \ge d(i)\ .
\end{array}\]
In particular, we note that if $d(i)=1$ for all $i\in\{1,\ldots,n\}$, then we obtain the well known \emph{minimum set covering problem}.
\hfill$\diamond$
\end{definition}

The minimum set multi-covering problem plays a double role in this paper: (i) we reduce both MCPs to a minimum set multi-covering problem; and (ii) by polynomially reducing it to the rMCP we show the latter to be NP-hard.  A (computational) problem is said to be \emph{reducible in polynomial time} to another if there exists a procedure to transform the former to the latter using a polynomial number of operations on the size of its inputs. Such reduction is useful in determining the qualitative complexity class \cite{Garey:1979:CIG:578533} a particular problem belongs to. For instance, we will need the following result.

\begin{proposition}[\hspace{-0.01cm}\cite{Garey:1979:CIG:578533}]
Let $\mathcal P_A$ be an NP-hard problem. If there is a polynomial reduction from $\mathcal P_A$ to $\mathcal P_B$, from which a solution to $\mathcal P_A$ can be determined, then $\mathcal P_B$ is an NP-hard problem.\hfill $\diamond$\label{proposition2}
\end{proposition}

Similarly, the minimum set covering problem is used in the present paper to show the NP-completeness of the MCP, by considering the following result.

\begin{lemma}[\hspace{-0.01cm}\cite{Garey:1979:CIG:578533}]
Let $\mathcal P_A$ and $\mathcal P_B$ be two NP-hard problems, and $\mathcal P_A^d$ and $\mathcal P_B^d$ be their decision versions, respectively.  
If a problem $\mathcal P_A$ is polynomially reducible to $\mathcal P_B$ (or equivalently, their decision versions) and $\mathcal P_B$ is polynomially reducible to $\mathcal P_A$ (or equivalently, their decision versions), then both $\mathcal P_A^d$ and $\mathcal P_B^d$ are NP-complete.\hfill $\diamond$\label{proposition1}
\end{lemma}

Now, given an arbitrary LTI system~\eqref{dyn}, we will focus on the following controllability tests.

\begin{theorem}[\hspace{-0.01cm}\cite{Hespanha09}]{(PBH test for controllability using eigenvalues)}\label{popovval} The system described in ~\eqref{dyn}  is controllable \emph{if and only if}
	$\text{rank}\left(\left[\begin{array}{lr}A-\lambda \text{I}_n & B\end{array}\right]\right)=n\text{ for all }\lambda\in\mathbb{C}.$
\hfill $\diamond $
\end{theorem}

%\vspace{-0.2cm}

In fact, it suffices to verify the criterion of Theorem~\ref{popovval} for each $\lambda\in \sigma (A)$. Observe that Theorem~\ref{popovval} provides a polynomial method to check the controllability of an LTI system since for each eigenvalue $\lambda$ of $A$ only the computation of the rank of $[A-\lambda I_n\ B]$ is required. Nevertheless, it does not provide any immediate information about which entries of $B$ should be different from zero and with what particular values such that the rank condition is ensured. That is, verifying if a $B$ is a solution can be achieved in P, so the controllability problem is in NP. Therefore, a naive usage of the PBH eigenvalue test would lead to a strictly combinatorial procedure for solving the MCP. Instead, we can consider the PBH test for controllability using eigenvectors.

%\vspace{-0.1cm}

\begin{theorem}[\hspace{-0.01cm}\cite{Hespanha09}]{(PBH test for controllability using eigenvectors)}\label{popov} Given \eqref{dyn}, the system is not controllable \emph{if and only if} there exists a left-eigenvector $v$ of $A$ such that $v^\dagger B= 0$.\hfill$\diamond$
\label{popovvalvector}
\end{theorem}

%
%Let $E^l_{\lambda_i}$ be the left-eigenspace  associated with the eigenvalue $\lambda_i$. Then, the left-eigenspace associated with a matrix $A$, with $\lambda_i\in\sigma(A)=\{\lambda_1,\ldots,\lambda_p\}$ is given by $E=E^l_{\lambda_1}\oplus \ldots \oplus E^l_{\lambda_p}$, where $\oplus$ is the direct sum of the subspaces, whose corresponding algebraic multiplicities we denote by $\alpha_1,\ldots,\alpha_p$. Then, $E$ is a left-eigenbasis if $\alpha_1+\ldots+\alpha_p=n$, which implies that $\text{dim}(E^l_{\lambda_i})=\alpha_i$. Therefore, for each $\alpha_i>1$ it follows there exist $\alpha_i$ linear independent left-eigenvectors that span $E^l_{\lambda_i}$. In particular, consider  $A$ to be the $n\times n$ matrix given by
%\[\scriptsize \mathbbm{1}_n=\left[\begin{array}{ccccc}
%1&1&\cdots&1 \\
%1&1&\cdots&1 \\
%\vdots &\vdots &\ddots &\vdots \\
%1&1&\cdots&1
%\end{array}\right]_{n\times n}.\] 
%Then, a subset of possible (and distinct) left-eigenvectors associated with the eigenvalue $0$ consists of all possible eigenvectors whose entries sum up to zero, which implies that the set of all possible eigenvectors increases exponentially in~$n$. 
% Thus, rendering the use of the PBH test for controllability using eigenvectors  computationally intractable to verify controllability; more precisely,  it requires computing the entire set of left-eigenvectors, and posterior verification if each one is orthogonal to the input vector. Nonetheless, in this paper we show that the PBH eigenvector test provides an insight about the design of $B$.

To relate our results with the ones from structural systems and further understand the advantages and drawbacks of this approach, we will introduce the structural counterpart of the MCP, the \emph{minimal structural controllability problem}~(MSCP). But first, we need to review the structural counterpart of controllability~\cite{dionSurvey}.

\begin{definition}[\hspace{-0.01cm}\cite{dionSurvey}]\label{strucContr} {(Structural Controllability)}
Given an LTI system \eqref{dyn} with sparseness given by $(\bar A,\bar B)$, with $\bar A\in\{0,\star\}^{n\times n}$ and $\bar B\in\{0,\star\}^{n\times p}$, where the entries correspond to fixed zeros and free real parameters, the pair $(\bar{A},\bar B)$ is said to be structurally controllable if there  exists a controllable pair $(A,B)$, with the same sparseness as $(\bar A, \bar B)$.\hfill$\diamond$
\end{definition}

 In fact, a stronger characterization of structural controllability holds as stated in the following proposition.
%\vspace{-0.1cm}

\begin{proposition}[\hspace{-0.01cm}\cite{Reinschke}]\label{zeroLebStruct} For a structurally controllable  pair  $(\bar A, \bar B)$, the numerical realizations $(A,B)$ with the same  sparseness as $(\bar A,\bar B)$ that are non-controllable  are described by a proper variety in $\mathbb{R}^{n\times n}\times\mathbb{R}^{n\times p}$. In other words, almost all realizations respecting the structural  pattern of a structurally controllable pair  are controllable.\hfill $\diamond$
\end{proposition}

By almost all realizations, we mean that at most a set with zero Lebesgue measure will lead to numerical realizations that do not ensure  controllability.

Subsequently, the  MSCP is posed as follows: given the structural matrix  $\bar A\in\{0,\star\}^{n\times n}$ associated with the dynamics' matrix $A$,  find $\bar B$ such that
\vspace{-0.4cm}

\begin{equation}
	\begin{array}{cc}\bar B= \underset{\bar B' \in \{0,\star\}^{n\times n}}{\arg \min} & \quad \| \bar B ' \|_0 \\  \quad \text{ s.t. } &\quad (\bar A,\bar B') \text{ is structurally controllable.}
	\end{array}
\label{strucmincontprob}
\end{equation}

Now, note that, by Definition \ref{strucContr}, a pair $(A,B)$ is controllable {only if} the corresponding structural pair $(\bar A,\bar B)$ is structurally controllable.  Therefore, it is natural  first to characterize all the sparsest structures  of  input vectors that ensure structural controllability, i.e., solutions to \eqref{strucmincontprob}.  
In particular, as a consequence of Proposition~\ref{zeroLebStruct}, we have the following result which links the MCP to its structural counterpart.
%\vspace{-0.2cm}

\begin{proposition}[\hspace{-0.01cm}\cite{PequitoJournal}]\label{probstrucvsmin}
Given $A\in\mathbb{R}^{n\times n}$, a solution  $B\in \mathbb{R}^{n\times p}$ for the MCP and  a numerical realization $B'\in \mathbb{R}^{n\times p}$ of a solution to the MSCP associated with the structural matrix $\bar{A}$, we have 
\vspace{-0.2cm}
\[
\|B\|_0\ge \|B'\|_0.
\] 
\vspace{-0.6cm}

\noindent More generally, for each $B$ that solves the MCP, there exists a solution $\bar{B}'$ of the MSCP such that

 %\vspace{-0.3cm}

 $$\bar B\geq \bar B',$$
 \vspace{-0.7cm}
 
\noindent where $\bar{B}$ and $\bar B'$ denote the structural matrix associated with $B$ and $B'$, respectively. 
 Conversely, given a structural matrix $\bar{A}$ and a solution $\bar{B}'$ to the MSCP, for almost all numerical instances $A$ satisfying the structural pattern of $\bar{A}$, then almost all numerical instances  satisfying the structural pattern  of $\bar{B}'$ are solutions to the MCP associated with $A$.
\hfill $\diamond$
\end{proposition}
\vspace{-0.1cm}

\vspace{-0.2cm}

\section{Minimum Controllability Problems}\label{MCPSection}

In this section, we provide the main results of this paper. In Section~\ref{exactSolSection}, we show that the MCP can be exactly solved in two steps: (i) Polynomial reduction of the structural optimization problem~\eqref{mcpStatement} to a set-covering problem (Algorithm~\ref{buildsets}); and (ii)  determine a numerical parametrization of an
	input matrix $B$ with specified input structure $\bar B$, in a deterministic polynomial fashion (Algorithm~\ref{procedure}). Further, by sequentially performing the two algorithms, we are `decoupling' the discrete and continuous properties of the MCP without losing optimality (Theorem~\ref{soundness}). In other words, we treat separately the identification of the matrix pattern $\bar B$ (discrete property) and the computation of a numerical realization encompassing the $\bar B$ pattern, and ensuring controllability of the system (continuous property).  In Section~\ref{rMCPsec}, we show that rMCP is NP-hard (Theorem~\ref{NPhardrMCP}), and a similar procedure to that used to solve MCP is followed. More specifically,  we determine the sparsity of an input matrix by polynomially reducing the problem to a minimum set multi-covering problem (Theorem~\ref{mainTheorMCP}), and this can later be used to characterize  the solutions to rMCP (Theorem~\ref{generalSol2rMCP}).

%This needs to be more clear: 
%One option is to complement with  the titles of the Algorithms:
%1. Polynomial reduction of the structural optimization problem (2) to a set-covering problem
%2. Determines a numerical parametrization of an
%input matrix B with specified input structure B

Complementary to the solutions to the MCPs, in Section~\ref{complexitySec}, we show that in fact the decision versions of MCP and rMCP (under Assumption~1) are NP-complete (Theorem~\ref{thmain}). Further, in Section~\ref{approxMCPsSec}, because the MCPs are NP-hard, we discuss a possible approach that leverage existing polynomial algorithms used to determine good approximations of the solutions to the minimum set multi-covering problem (for instance, Algorithm~\ref{algorithm3}). Subsequently, we argue that the approximate solution warrants some optimality guarantees (Theorem~\ref{complexityAlg3}). Finally, in Section~\ref{numericalAnalysis}, we explore some numerical implications of waiving Assumption~2.

 \subsection{A Characterization of the MCP Solution}\label{exactSolSection}
 
In this section, we present a systematic method to obtain a solution to the MCP problem.   
First, we show that given a left-eigenbasis of the   dynamics'  matrix $A$, it is possible to polynomially reduce the MCP to the minimum set covering problem. This reduction assumes that we only have a single effective input to actuate the system, i.e., the input matrix has a single non-zero column. Notice that a feasible solution always exist because $A$ is simple. Subsequently, we say that the input vector is a solution to the MCP if the input matrix obtained consists of one effective input associated with that input. Further, in Theorem~\ref{generalSol2MCP}, we show that this can be done without loss of generality. The reduction is achieved by exploiting the PBH eigenvector criterion (Theorem~\ref{popovvalvector}) for controllability. More precisely, the reduction  is obtained  in two steps:  first, we provide a necessary condition on the structure $\bar{b}$ of the sparsest input vector $b$ (see   Lemma~\ref{lemma1}), which is obtained by formulating a minimum set covering problem (see Algorithm~\ref{buildsets}) associated with the structure (i.e., location of non-zero entries) of the left-eigenvectors of the dynamics'  matrix $A$. Secondly, we show that  a possible numerical realization  of $\bar{b}$ which solves the MCP may be generated using a polynomial construction (Algorithm~\ref{procedure}). Both algorithms (Algorithm~\ref{buildsets} and Algorithm~\ref{procedure}) have polynomial complexity in the number of state variables (Theorem~\ref{soundness}). Further, the sequential use of these algorithms  provides a systematic solution to the MCP (see Theorem~\ref{exactSolTheorem}).

The first set of results provides necessary conditions on the structure  that an input vector $b$ must satisfy to ensure controllability of $(A,b)$,  and a polynomial complexity procedure (Algorithm~\ref{buildsets}) that reduces the problem of obtaining such necessary structural patterns to a minimum set covering problem.

\begin{lemma}\label{lemma1}
	Given a collection of non-zero vectors $\displaystyle\{\bar v^j\}_{j\in\mathcal{J}}$ with $\bar v^j\in\{0,\star\}^n$,  
the procedure of finding  $\,\bar b^*\in\{0,\star\}^n$ such that 
	\begin{equation}\label{label3}\begin{array}{rc}\bar b^*=\underset{\bar b\in\{0,\star\}^n}{\arg \min}& \|\bar b\|_0\\
	\text{s.t. }\quad & \bar v^j\cdot \bar b\neq 0,\text{ for all }j\in\mathcal{J}
	\end{array}
\end{equation}
	is polynomially (in $|\mathcal J|$ and $n$) reducible to a minimum set covering problem with universe $\mathcal{U}$ and a collection $\mathcal{S}$ of sets by applying Algorithm~\ref{buildsets}.
	\hfill$\diamond$
\end{lemma}

\begin{algorithm}
\small

\textbf{Input:} $\{\bar v^j\}_{j\in \mathcal J}$, a collection of $|\mathcal J|$  vectors in $\{0,\star\}^n$.

\textbf{Output:} $\mathcal S=\{\mathcal S_i\}_{i\in\{1,\ldots,n\}}$ and $\mathcal U$, a set of $n$ sets and the universe of the sets, respectively.

\begin{algorithmic}[1]

\STATE \textbf{set} $\mathcal S_i=\{\}$ for $i=1,\ldots,n$ 
	
\STATE \textbf{for} $j=1,\ldots,|\mathcal J|$ 
                
               \qquad \textbf{for} $i=1,\ldots,n$ 

				\qquad\qquad \textbf{if} $\bar v_i^j\neq 0$ \textbf{then}
				
				\qquad\qquad\qquad $\mathcal S_i=\mathcal S_i\cup\{j\}$;
				
				\qquad\qquad \textbf{end if}
    
    \qquad \textbf{end for}

 \textbf{end for}

\STATE \textbf{set} $\mathcal S=\{\mathcal S_1,\hdots,\mathcal S_{n}\}$ and  $\mathcal U=\displaystyle\bigcup_{i=1}^{n}\mathcal S_i$.

\end{algorithmic}
\caption{Polynomial reduction of the structural optimization problem~\eqref{mcpStatement} to a set-covering problem}	
\label{buildsets}
\end{algorithm}

Next, we show that given the structure obtained in Lemma~\ref{lemma1}, almost all possible real numerical realizations lead to a vector $b\in\mathbb{R}^n$  that is a solution to the MCP.

\begin{theorem}\label{acharb}
Let $\{v^{i}\}_{i\in\mathcal{J}}$ to be the set of left-eigenvectors of $A$, and $\bar b$ a solution to~\eqref{label3}. Then, almost all numerical realizations $b$ of $\bar{b}$ are solutions to the MCP.
\hfill$\diamond$
\end{theorem}

Observe that Theorem~\ref{acharb} differs from the converse result in Proposition~\ref{probstrucvsmin}  in a subtle, yet important, manner which we describe in the following remark.

\begin{remark}\label{almostAllinputs}
The converse result in Proposition~\ref{probstrucvsmin} about the generic properties that characterize structural controllability shows  that almost all parameters of both dynamics and input matrices satisfying a given structural pattern are controllable. Although, in Theorem~\ref{acharb} the dynamics' simple matrix $A$ is fixed, i.e., a numerical instance with specified structure, and density arguments are provided to the numerical realizations of the input vector with certain structure ensure controllability of the system.\hfill $\diamond$
\end{remark}

Although Theorem~\ref{acharb} ensures that almost all parameterizations provide a feasible solution to the MCP, we need to determine one parameterization that guarantees controllability. Toward this goal, in Algorithm~\ref{procedure}, we present an efficient algorithm to obtain such parameterization. The correctness and computational complexity of the algorithm is provided in the next result.

\begin{theorem}\label{soundness} Algorithm~\ref{procedure} is correct and has complexity $\mathcal O(|\mathcal J|)$, where $|\mathcal J|$ is the size of the collection of vectors given as input to the algorithm.\hfill $\diamond$ \end{theorem}

\begin{algorithm}
	
\textbf{Input:} $\{v^j\}_{j\in \mathcal J}$, a collection of $|\mathcal J|$  complex vectors, and $\bar B\in \{0,\star\}^{n\times m}$.

\textbf{Output:} $B^*\in \mathbb{R}^{n\times m}$ solution to~\eqref{label4}.

\begin{align*}
B^* =\arg\min_{B\in \mathbb{R}^{n\times m}}  &\qquad \qquad 0\\
\text{s.t. } &\qquad {(v^j)}^\dagger B>0, \quad j\in \mathcal J\\
&\qquad  B_{l,k}=0 \text{ if } \bar B_{l,k}=0, \ \ l,k=1,\ldots,n
\end{align*}

\caption{Determines a numerical parametrization of an input matrix $B$ with specified  input structure $\bar B$ \vspace{0.1cm}}
\label{procedure}
\end{algorithm}

Whereas Algorithm~\ref{procedure} provides an efficient formulation that enables to retrieve a possible parametrization ensuring controllability, one can easily extend this framework to more general scenarios aiming to capture some additional control metrics of interest, for instance, the controllability energy. This extensions are described in further detail in the following remark.

\begin{remark}
Suppose the objective function in Algorithm~\ref{procedure} is given by $f(B)$. Then, this  can be chosen to satisfy additional design constraints. For instance, $f(B)=c^\intercal B \mathbf{1}$, where  $c$ could capture an actuation cost, i.e., entry $c_i$ captures how desirable is to actuate $x_i$, and $\mathbf{1}$ is a vector of ones with appropriate dimensions. Subsequently, one may need additional constraints such that the total actuation budget $r$ available is bounded, for instance, $|f(B)|\le r$ and $B_{i,j}\ge 0$ to avoid negative entries that will restrain the objective goal.  Alternatively, $f(B)$ can also be considered to be nonlinear, while capturing control-theoretic properties; in particular, it can be a function of the controllability Grammian~\cite{PasqualettiJournalControl}, with some appropriate constraints to ensure the problem to be well defined. \hfill $\diamond$
\end{remark}

Next, we show that the sparsest vector pattern given by Lemma~\ref{lemma1}, together with Algorithm~\ref{procedure}, leads to a numerical realization that is a solution to the MCP.

\begin{lemma}\label{lemma4}
	Given $\displaystyle\{v^i\}_{i\in\mathcal{J}}$ with $v^i\in\mathbb{C}^n$, the procedure of finding $ b^*\in\mathbb{R}^n$ such that 
	\begin{equation}\label{label4}\begin{array}{rc}b^*=\underset{b\in\mathbb{R}^n}{\arg \min} &\|b\|_0
	\\
	\text{s.t. } & v^i\cdot b\neq 0,\text{ for all }i\in\mathcal{J},
	\end{array}\end{equation}
	is polynomially (in $|\mathcal J|$ and $n$) reducible to a minimum set covering problem (provided by Algorithm~\ref{buildsets}), with numerical entries determined using Algorithm~\ref{procedure}.
	\hfill$\diamond$
\end{lemma}

Now, we state one of the main results of the paper.

\begin{theorem}
 The solution to the MCP can be determined by first identifying the sparsity of the input vector as in Lemma~\ref{lemma1}, followed by determining the numerical realization of the non-zero entries using Algorithm~\ref{procedure}.
\hfill $\diamond$
\label{exactSolTheorem}
\end{theorem}

Finally, we characterize  the sparsity solutions to the MCP besides those described by a single effective input.

%%%%%%%%%

\begin{theorem}
Let $b\in\mathbb{R}^{n}$ be a solution to the MCP as described in Theorem~\ref{exactSolTheorem}, $\bar b$ its sparsity and $\mathcal N\subset\{1,\ldots,n\}$ the indices where $\bar b$ is non-zero, i.e., $\mathcal N=\{i: \bar b_i=\star, \text{ and }i=1,\ldots,n\}$. If $\bar B \in \{0,\star\}^{n\times n}$ has exactly one non-zero entry in the \mbox{$i$-th} row, where $i\in \mathcal N$, then  the output $B\in \mathbb{R}^{n\times n}$  of Algorithm~\ref{procedure}, when $\bar B$ and the  left-eigenbasis of $A$ are considered, is a solution to the MCP.  \hfill $\diamond$
\label{generalSol2MCP}
\end{theorem}

In particular, from Theorem~\ref{generalSol2MCP}, we obtain the following result regarding the scenario where  every effective input actuates a single state variable, which we refer to as \emph{dedicated} inputs.

\begin{corollary}\label{dedicatedSolMCP}
Let $b\in\mathbb{R}^{n}$ be a solution to the MCP as described in Theorem~\ref{exactSolTheorem}, $\bar b$ its sparsity and $\mathcal N\subset\{1,\ldots,n\}$ the indices where $\bar b$ is non-zero, i.e., $\mathcal N=\{i: \bar b_i=\star, \text{ and }i=1,\ldots,n\}$. If $\bar B \in \{0,\star\}^{n\times n}$ has exactly one non-zero entry in the \mbox{$i$-th} row and each column, where $i\in \mathcal N$, then  the output $B\in \mathbb{R}^{n\times n}$  of Algorithm~\ref{procedure}, when $\bar B$ and the  left-eigenbasis of $A$ are considered, is a dedicated solution to the MCP, i.e., every effective input actuates a single state variable.
\hfill $\diamond$
\end{corollary}

\subsection{On the Exact Solution of the Robust Minimal Controllability Problem}\label{rMCPsec}

Now, we study the rMCP, by first showing that this is an NP-hard problem (Theorem~\ref{NPhardrMCP}). Then,  similarly to the previous subsection, we first show that a particular subclass of input matrices is a solution to this problem. More specifically, we  characterize the dedicated solutions to the rMCP (Theorem~\ref{mainTheorMCP}), and, subsequently, we provide a characterization of the solution to the rMCP in Theorem~\ref{generalSol2rMCP}.

\begin{theorem}\label{rMCPNP}
The rMCP is NP-hard.\hfill $\diamond$
\label{NPhardrMCP}
\end{theorem}

Now, similar to the reduction proposed from MCP to the set covering problem, we can characterize the dedicated solutions to the rMCP by considering a set multi-covering problem as stated in the next result.

\begin{theorem}
Let $v^1,\ldots, v^n$ be a left-eigenbasis of $A$, and $s$ the number of possible input failures. Further, consider the set multi-covering problem $(\{\mathcal S_1,\ldots, \mathcal S_{(s+1)n}\}$, $\mathcal U\equiv \{1,\ldots,n\}; d)$, where the demand is $d(i)=s+1$  for $i\in \mathcal U$, and $\mathcal S_k=\{j: [v^j]_{l}\neq 0, \text{ and } l-1=k\mod n\}$ for $k\in\mathcal K\equiv\{1,\ldots,(s+1)n\}$. Then, the following statements are equivalent:
\begin{enumerate}
\item[(i)] $\mathcal M^*$ is a solution to the set multi-covering problem $(\{\mathcal S_1,\ldots, \mathcal S_{(s+1)n}\},\mathcal U\equiv \{1,\ldots,n\}; d)$;
\item[(ii)]  $B_n(\mathcal M^*)$ is a dedicated solution to rMCP, where $[B_n(\mathcal M^*)]_{i,l}=1$  for  $l=i\mod n$ and $i\in\mathcal M^*\subset \mathcal K$, and zero otherwise.\hfill $\diamond$
\end{enumerate}
\label{mainTheorMCP}
\end{theorem}

\begin{remark}\label{CasoNotavel}
	A matrix $B_n(\mathcal M')$ described by the concatenation of $(s+1)$ solutions to the MCP achieves feasibility to the rMCP, but it is not necessarily an optimal solution to the rMCP. In  Section~\ref{subSecrMCPexample}, we provide an example where the concatenation of solutions is not a solution to the rMCP.\hfill $\diamond$
\end{remark}

In Theorem~\ref{mainTheorMCP}, we provided a characterization of dedicated solutions to the rMCP.  In particular, we notice that the solution may require that several non-zero entries in a row of a dedicated solution are considered. In other words, the same state variable needs to be actuated by different actuators to ensure robustness for $s$ input failures. 

Next, we characterize the solutions of the rMCP, i.e.,  not only the ones that are dedicated.  Towards this goal, we need to introduce the following \emph{merging} procedure. Let two distinct effective inputs $i$ and $j$, associated with two non-zero columns of the input matrix, $b^i$ and $b^j$, be such that they share no non-zero entry $k$, i.e., $[b^i]_k\neq [b^j]_k$ for $k\in\{1,\ldots,n\}$. These two inputs are said to be merged into one input $b^{i'}$, where $[b^{i'}]_k=[b^i]_k$ when $[b^i]_k\neq 0$, and $[b^{i'}]_k=[b^j]_k$ when $[b^j]_k\neq 0$, for $k\in\{1,\ldots,n\}$. Further, it is implicitly assumed that $b^{i'}$ takes the place of $b^i$ and $b^j$ is set to zero. In other words, the effective input $i$ is associated with $b^{i'}$ and the effective input $j$ is discarded.

\begin{theorem}\label{generalSol2rMCP}
Let $B_n(\mathcal M^*)\in\mathbb{R}^{n\times (s+1)n}$ be a dedicated solution to the rMCP as described in Theorem~\ref{mainTheorMCP}. In addition, let $\bar B \in \{0,\star\}^{n\times (s+1)n}$ be the sparsity of the matrix resulting of the merging procedure between any of the effective inputs in $B_n(\mathcal M^*)$. Then, the matrix $B\in \mathbb{R}^{n\times n}$  obtained using Algorithm~\ref{procedure}, with $\bar B$ and the  left-eigenbasis of $A$, is a solution to the~rMCP. 
\hfill$\diamond$
\end{theorem}

\subsection{Computational Complexity}\label{complexitySec}

In the previous subsections, we have mentioned that both MCPs are NP-hard. The NP-hardness assesses that a problem is at least as difficult as another NP-hard problem. In this subsection, we show that both MCPs are NP-complete, i.e., their decision versions are NP-complete. Therefore, we provide an interesting remark about NP-completeness class from results known in control systems. Also, it sets the grounds for the next subsection, where polynomial approximation algorithms (that obtain a suboptimal solution to the set multi-covering problem) are leveraged to obtain approximate solutions to the MCP and rMCP.

\begin{theorem}\label{thmain}
	The MCP  and rMCP are NP-complete.
%\vspace{-0.6cm}
	\hfill$\diamond$
\end{theorem}
\vspace{-0.2cm}

Additionally, Theorem~\ref{thmain}  leads to the following interesting observation.
\begin{remark}\label{NPprobUnlikely}
By Proposition~\ref{probstrucvsmin} (the converse part), it follows that a solution of the MCP almost always coincides with a numerical realization of a solution to an associated minimal structural controllability. 
Combining this with the fact that the MCP is NP-complete when the eigenvalues of $A$ are simple (see Theorem~\ref{thmain}), it follows that the set of  simple dynamics' matrices that lead to NP-complete problems has zero Lebesgue measure. \hfill$\diamond$
\end{remark}

As stated in Theorem~\ref{thmain}, the condition that the  matrices $A$ be restricted to have simple eigenvalues, is, in fact, necessary in a sense for the proposed reduction of the MCP to the minimum  set covering problem to be polynomial in $n$. 
This fact is explored in the next remark.

\begin{remark}
The proposed reduction from the MCP to the minimum set covering problem is polynomial in $\max(|\mathcal{J}|,n)$, where $|\mathcal{J}|$ denotes  the number of left-eigenvectors. Nevertheless, because the number of left-eigenvectors can grow exponentially, it follows that the proposed reduction cannot be used to show that the decision version of the  (general) MCP is NP-complete.  However, this does not imply  that the decision version of the  MCP for arbitrary dynamics'  matrices (i.e., when $A$ is not restricted to have simple eigenvalues)   is not NP-complete, which remains an open question.\hfill $\diamond$
\end{remark}

Finally, we notice that the fact that a problem is NP-hard, it does not mean that all instances are not solved polynomially; notwithstanding, these can be solved exactly~\cite{exactAlgSetMultiCover,DPAlgSetMultiCover}. Furthermore, the NP-completeness stated in Theorem~\ref{thmain}, allows us to consider the subclasses of the set multi-covering problem that are known to be polynomially solvable, to identify polynomially solvable subclasses of the MCPs. This enables a new characterization of solutions to the question posed in~\cite{controlLaplac}, regarding the existence of polynomial algorithms exist to determine controllable graph structures. In particular, we notice that  in several of these cases, the graphs are associated with dynamics' matrices that are simple -- the case explored in this present paper. Alternatively,  by the proposed construction, if the set multi-covering problem obtained possess additional structure, then this can be leveraged to use polynomial algorithms to approximate the solutions with close-to-optimal solutions, as we discuss in the next subsection.

\subsection{Polynomial Approximations to the Solution of the Minimal Controllability Problems}\label{approxMCPsSec}

As a consequence of Theorem~\ref{thmain}, it follows that we can obtain polynomial approximations for both the multi-set covering problem and the rMCP. Notice that, in particular, a solution to the MCP can be obtained by considering that no input fails. Therefore,  in Algorithm~\ref{algorithm3},  we propose an algorithm that leverages the submodularity properties~\cite{BachSubmodular} of the set multi-covering properties to obtain a dedicated solution to the rMCP.   Submodularity properties ensure that the associated polynomial greedy algorithms have sub-optimality guarantees while performing well in practice~\cite{BachSubmodular}, see also Remark~\ref{remark6}. Subsequently, following a similar reasoning to that presented in Theorem~\ref{mainTheorMCP}, we can obtain the following result.

\begin{algorithm}
\small

\textbf{Input:} Left-eigenbasis $v^1,\ldots,v^n$ associated with $A\in\mathbb{R}^{n\times n}$ and the number of possible input failures $s$.

\textbf{Output:} Dedicated solution $B_n(\mathcal M')\in\mathbb{R}^{n\times (s+1)n}$.

\begin{algorithmic}[1]

\STATE Let $\mathcal S_1,\ldots,\mathcal S_{(s+1)n}$,  where  $\mathcal S_k=\{j: [v^j]_{l}\neq 0, \text{ and } l-1=k\mod n\}$ for $k\in\mathcal K\equiv\{1,\ldots,n(s+1)\}$.

\STATE \textbf{set} $\mathcal U^i=\emptyset$, with $i=1,\ldots, s$ $\triangleright$ denote the indices in $\mathcal U$ that are covered $i$ times and the indices of the sets covering them, respectively. 

\STATE \textbf{set} $\mathcal J=\emptyset$

\STATE \textbf{for} $i=1,\ldots,s+1$ 

	\qquad \textbf{set} $\mathcal U^i=\{k: |\{k\in\mathcal U: k\in \mathcal S_j, j\in \mathcal J\}|\ge i  \}$ $\triangleright$ the indices that are already covered by at least $i$ sets

\STATE\qquad\textbf{while} $\mathcal U^i\neq \mathcal U$

\qquad\qquad \textbf{select} $\mathcal S_j$ with largest number of indices in $\mathcal U\setminus \mathcal U^i$

 \qquad\qquad \textbf{set}  $\mathcal J\leftarrow\mathcal J \cup \{j\}$

\qquad\qquad \textbf{set} $\mathcal U^i\leftarrow \mathcal U^i\cup \mathcal S_j$
    
    \qquad \textbf{end while}

%	\qquad \textbf{set} $\mathcal J\leftarrow  \mathcal J\cup \mathcal J^j$;

 \textbf{end for}

\textbf{set} $\mathcal M'\gets\mathcal J$;
\end{algorithmic}
\caption{Approximate Solution to the rMCP}
\label{algorithm3}
\end{algorithm}

\begin{theorem}\label{complexityAlg3}
	The matrix $B_n(\mathcal M')$ obtained using Algorithm~\ref{algorithm3}, with $\bar B$ and the  left-eigenbasis of $A$,  is a feasible solution to the rMCP.  Further, the computational complexity of Algorithm~\ref{algorithm3} is  $\mathcal O(s  n)$, and it ensures an approximation optimality bound of $\mathcal O(\log{n})$.
\hfill$\diamond$  		
\end{theorem}

\begin{remark}\label{remark6}
	 Algorithm~\ref{algorithm3} produces suboptimal solutions that are often optimal solutions to the rMCP, as illustrated in Section~\ref{subsecAillustrativeExample}. The practical performance, together with the linear computational complexity motivated the choice of such procedure.   
Nonetheless, the information on the structure of the left-eigenvectors, or equivalently, the structure of the sets in the  set multi-covering problem,  can be leveraged to obtain better approximations, for instance, see~\cite{HochbaumApproximationSC,Bronnimann1995}. In particular, the approximation algorithm from~\cite{Bronnimann1995} outperforms the majority of the known approximation algorithms if the number of elements of the largest set is small. The authors obtained an approximation optimality bound of $\mathcal O(d\log{d c})$, where $c$ is the size of an optimal solution and $d$ the  number of elements of the largest set, and its computational complexity is $\mathcal O(c n\log{\frac{n}{c}})$. 
Further, ~\cite{Guy} extends the latter results by using a linear programming relaxation, which has comparable computational complexity, but with a better approximation ratio that is smaller by a constant factor. Also, in~\cite{Guy} the approach is directly applicable to set multi-covering problems, required to determine the solution to the rMCP.
\hfill$\diamond$  
\end{remark}

Finally, by invoking Theorem~\ref{generalSol2rMCP}, we obtain the following result.

\begin{corollary}\label{generalSol3rMCP}
Let $B_n(\mathcal M')\in\mathbb{R}^{n\times (s+1)n}$ be a dedicated solution to the rMCP as described in Theorem~\ref{complexityAlg3}. In addition, let $\bar B \in \{0,\star\}^{n\times (s+1)n}$ be the sparsity of the matrix resulting of the merging procedure between any of the effective inputs in $B_n(\mathcal M')$. Then, the matrix $B\in \mathbb{R}^{n\times n}$ obtained using Algorithm~\ref{procedure}, with $\bar B$ and the  left-eigenbasis of $A$, achieves feasibility to the rMCP and is computed in polynomial time.
\hfill$\diamond$
\end{corollary}

\subsection{Numerical and Computational Remarks}\label{numericalAnalysis}

 Now, for the sake of completeness, we discuss the implications of waiving Assumption~2 and the impact on  the input vector in the MCP. The results trivially extend to the general solution to the MCPs. Towards this goal, we need the following result from \cite{Pan:1999:CME:301250.301389}.

\begin{theorem}[\hspace{-0.01cm}\cite{Pan:1999:CME:301250.301389}]\label{eigen}
Let $A\in \mathbb{C}^{n\times n}$ be a matrix with simple eigenvalues. The deterministic arithmetic complexity of finding the eigenvalues and the eigenvectors of $A$ is bounded by $\mathcal{O}\left(n^3\right)+t\left(n,m\right)$ operations, where $t(n,m)=\mathcal{O}\left(\left(n\log^2n\right)\left(\log m+\log^2 n\right)\right)$, for a required upper bound of $2^{-m}\|A\|$ on the absolute output error of the approximation of the eigenvalues and eigenvectors of $A$ and for any fixed matrix norm $\|\cdot\|$.\hfill$\diamond$
\end{theorem}

More precisely, Theorem~\ref{eigen} states that in practice, only a numerical approximation of the left-eigenbasis is possible in polynomial time. In this case, let $\epsilon=2^{-m}\|A\|$ be as in Theorem~\ref{eigen}, then the results stated in Lemma~\ref{lemma1} and Lemma~\ref{lemma4} (see also Algorithm~\ref{buildsets} and Algorithm~\ref{procedure}) can only be used in an $\epsilon$-\emph{approximation} of the left-eigenbasis of the dynamics' matrix. Therefore, the  $\epsilon$-approximation of the left-eigenbasis may lead to the following issues: 

 (i) an entry in the left-eigenvectors is considered as zero, where in fact it can be some non-zero value that (in norm) is smaller then $\epsilon$. Consequently, the sets generated using Algorithm~\ref{buildsets} (see also Lemma~\ref{lemma1}) do not contain the indices associated with those non-zero entries. Thus, additional sets need to be considered to the minimum set covering, which implies that the structure of the input vector may contain more non-zero entries than the sparsest input vector that is a solution to the MCP. In other words, we obtain an over-approximation of the sparsest input vector that is a solution to the MCP.

 (ii)  an entry in the left-eigenvectors in the \mbox{$\epsilon$-\emph{approximation}} of the left-eigenbasis is non-zero. Then, it does not represent an issue when computing the structure of the input vector as described in Lemma~\ref{lemma1} (see also Algorithm~\ref{buildsets}), but it can represent a problem when determining the numerical realization by resorting to Algorithm~\ref{procedure}. Nonetheless, by Theorem~\ref{acharb} it follows that such issue is unlikely to occur.

To undertake a deeper understanding of which entries fall in the first issue presented above, several methods to compute eigenvectors can be used and solutions posteriorly compared, see~\cite{eigenvectorsBOOK} for a survey of the different methods and computational issues associated with those.

\section{Illustrative examples}\label{secCompStrucvsNonStruct}

In this section, we provide some examples that illustrate the main results of the paper.

\subsection{Minimal Controllability Problem}\label{subsecAillustrativeExample}

To illustrate the first main result of this paper, i.e., to determine  a solution to $\mathcal P_1$, consider the dynamics' matrix $A$ given by
\begin{equation}\label{expa} A=
	\left[
	\begin{array}{ccccc}
	 6 & -3 & 3 & 2 & -1 \\
	 0 & 8 & 0 & 0 & 0 \\
	 4 & 3 & 7 & 2 & 1 \\
	 0 & 0 & 0 & 6 & 0 \\
	 -4 & -3 & -3 & -2 & 3 \\
	\end{array}
	\right]
,\end{equation}
where $\sigma(A)=\{2,4,6,8,10\}$ consists of distinct eigenvalues, so the matrix $A$ is simple and  the results in Section~\ref{exactSolSection} are applicable. Consequently, to obtain the solution to the MCP, we first   compute the  left-eigenvectors of $A$ that  are as follows: $ v^1=[\begin{array}{ccccc}1&1&0&0&1\end{array}]^\intercal$, $ v^2=[\begin{array}{ccccc}0&0&1&0&1\end{array}]^\intercal$, $ v^3=[\begin{array}{ccccc}0&0&0&1&0\end{array}]^\intercal$, $ v^4=[\begin{array}{ccccc}0&1&0&0&0\end{array}]^\intercal$ and $ v^5=[\begin{array}{ccccc}1&0&1&1&0\end{array}]^\intercal$. Therefore, their structures are given by {$\bar v^1=[\begin{array}{ccccc}\star&\star&0&0&\star\end{array}]^\intercal$, $\bar v^2=[\begin{array}{ccccc}0&0&\star&0&\star\end{array}]^\intercal$, $\bar v^3=[\begin{array}{ccccc}0&0&0&\star&0\end{array}]^\intercal$, $\bar v^4=[\begin{array}{ccccc}0&\star&0&0&0\end{array}]^\intercal$ and $\bar v^5=[\begin{array}{ccccc}\star&0&\star&\star&0\end{array}]^\intercal$}.
Using Algorithm~\ref{buildsets}, since $\bar v_i$ for $i=1,\hdots,5$, we obtain $\{\mathcal S_j\}_{j=1,\ldots,5}$, where the $j$-th set corresponds to the set of indices of the left-eigenvector which have a non-zero entry on the $j$-th position. In particular,  we obtain
	$\mathcal S_1  =  \left\{1,5\right\}$, 
	$\mathcal S_2  =  \left\{1,4\right\}$, 
	$\mathcal S_3  =  \left\{2,5\right\}$, 
	$\mathcal S_4  =  \left\{3,5\right\}$,  
	$\mathcal S_5  =  \left\{1,2\right\}$, and the universe set is given by
$\mathcal{U}=\displaystyle\bigcup_{i=1}^{n}\mathcal S_i=\left\{1,2,3,4,5\right\}.$
Now, it is easy to see that a solution to this minimum set covering problem is the set of indices $\mathcal I^*=\left\{2,3,4\right\}$, since $\mathcal U = \mathcal S_2\cup \mathcal S_3\cup \mathcal  S_4$ and there is no pair of sets, i.e.,  $\mathcal I'=\{i,i'\}$ with $i,i'\in\{1,\ldots,5\}$ such that $\mathcal U=\mathcal S_i\cup \mathcal S_{i'}$. Therefore, a possible structure of the vector $b$ that is a  solution to the MCP is
\begin{equation}\label{bbareq}\bar b=[\begin{array}{ccccc} 0 & \star &  \star & \star & 0\end{array}]^\intercal.\end{equation}
Additionally, to find the numerical parametrization of $b$, under the sparsity pattern of $\bar b$, we have to solve the following system with three unknowns: $b_2,b_3,b_4\neq0$ and $b_3+b_4\neq0$. By inspection, a possible choice is $b=[\begin{array}{ccccc} 0 & 1 &  1 & 1 & 0\end{array}]^\intercal$, but the numerical parametrization can be obtained by invoking Algorithm~\ref{procedure}, with the set of left-eigenvectors of $A$ given by $\{ v^j\}_{j\in\{1,\ldots,5\}}$ and the structure of $b$ given by $\bar b$ in \eqref{bbareq}. For the sake of completeness, we, the  controllability matrix is given by $$
\begin{array}{rcl}\mathcal C&=&[\begin{array}{ccccc} b & Ab &  A^2b & A^3b & A^4b\end{array}]\\[0.2cm] &=&\left[
\begin{array}{ccccc}
 0 & 2 & 44 & 608 & 7184 \\
 1 & 8 & 64 & 512 & 4096 \\
 1 & 12 & 120 & 1176 & 11520 \\
 1 & 6 & 36 & 216 & 1296 \\
 0 & -8 & -104 & -1112 & -11264 \\
\end{array}
\right],\end{array}
$$
and the rank$(\mathcal C)=5$, which implies that  $(A,b)$ is controllable. 

Observe that the single-input solution obtained with $b=[\begin{array}{ccccc} 0 & 1 &  1 & 1 & 0\end{array}]^\intercal$, can be immediately translated into a solution with two effective inputs, by invoking Theorem~\ref{generalSol2MCP}. In particular, two possible solutions are  $B=[\begin{array}{cc} b^1 & b^2\end{array}]$ with $b^1=[\begin{array}{ccccc} 0 & 1 &  1 & 0 & 0\end{array}]^\intercal$ and $b^2=[\begin{array}{ccccc} 0 & 0 &  0 & 1 & 0\end{array}]^\intercal$, and $B=[\begin{array}{ccc} b^1 & b^2 & b^3\end{array}]$ with $b^1=[\begin{array}{ccccc} 0 & 1 & 0 & 0 & 0\end{array}]^\intercal$, $b^2=[\begin{array}{ccccc} 0 & 0 & 1 & 0 & 0\end{array}]^\intercal$ and $b^3=[\begin{array}{ccccc} 0 & 0 &  0 & 1 & 0\end{array}]^\intercal$, where the latter is a dedicated solution. Alternatively,   if we consider for instance $B=[\begin{array}{cc} b^1 & b^2\end{array}]$ with $b^1=[\begin{array}{ccccc} 0 & 1 & 0 & 0 & 0\end{array}]^\intercal$ and $b^2=[\begin{array}{ccccc} 0 & 0 & -1 & 1 & 0\end{array}]^\intercal$, then $v^\intercal B=0$ for the left-eigenvector $v=[\begin{array}{ccccc}1&0&1&1&0\end{array}]^\intercal$ which renders the pair $(A,B)$ uncontrollable. Thus, as prescribed in Theorem~\ref{generalSol2MCP}, by invoking Algorithm~\ref{procedure}, one can obtain a new realization of $B$  that ensures controllability of $(A,B)$; for instance, the same $b^1$ and  $b^2=[\begin{array}{ccccc} 0 & 0 & \frac{12}{10} & 1 & 0\end{array}]^\intercal$.

In Section~\ref{approxMCPsSec}, a systematic polynomial approximation to the MCP can be obtained by considering the rMCP with the number of input failures equal to $s=0$. In fact, by doing so, one obtains the same sparsity to $b$, i.e., $\bar b$, as in the aforementioned example, and the subsequent analysis follows. Furthermore, we notice that the approximate solution is a solution to the MCP.

\subsection{Minimal Structural Controllability Problem}\label{sec:MSCP}

The solution to the MSCP considering $\bar{A}$ associated with $A$ in (7) is given by $\bar b'=[0 \ \star \ 0 \ \star \ 0]^\intercal$, see~\cite{PequitoJournal} for details.  Therefore, the structural controllability solution to the MSCP provides a strict lower bound on the number of state variables we should actuate with the input, i.e., the sparsity of the input vector (in accordance to Proposition~\ref{probstrucvsmin}). More precisely, we achieve structural controllability by actuating two variables (specifically $x_2$ and $x_4$), but in order to ensure controllability an additional state variable needs to be actuated, for instance, $x_3$ as obtained in Section~\ref{subsecAillustrativeExample}.  Therefore, structural controllability is necessary, but not sufficient,  to achieve controllability even when the matrix $A$ is simple. In particular, considering the converse part of Proposition~\ref{probstrucvsmin}, we note that the numerical values of the matrix $A$ fall into the set of zero Lebesgue measure (see also Proposition~\ref{zeroLebStruct}), where the  solution associated with the MSCP does not provide a solution to the~MCP. As a consequence, notice that Theorem~\ref{acharb} is different and stronger than Proposition~\ref{probstrucvsmin} (as observed in Remark~\ref{almostAllinputs}). More specifically, in Theorem~\ref{acharb} the matrix $A$ has fixed values and only the non-zero entries of $B$ are taken into account, whereas in Proposition~\ref{probstrucvsmin} both non-zero entries of $A$ and $B$ are considered not to be fixed.

 To sharpen the intuition behind these results and observations, we perturbed the matrix $A$ by adding a random uniform noise on the interval $[-10^{-10},10^{-10}]$ to each of its non-zero entries, which leads to a new matrix that we denote by  $A'$ (with the same structure as $A$). The structure of the left-eigenvector of the matrix $A'$ now becomes: $ \bar v^{\prime 1}=[\begin{array}{ccccc}\star& \star& \star& \star& \star\end{array}]^\intercal$, $ \bar v^{\prime 2}=[\begin{array}{ccccc}\star& \star& \star& \star& \star\end{array}]^\intercal$, $ \bar v^{\prime 3}=[\begin{array}{ccccc}0& 0& 0& \star& 0\end{array}]^\intercal$, $ \bar v^{\prime 4}=[\begin{array}{ccccc}0& \star& 0& 0& 0\end{array}]^\intercal$ and $ \bar v^{\prime 5}=[\begin{array}{ccccc}\star& \star& \star& \star& \star\end{array}]^\intercal$. Subsequently, building the sets for the minimum set covering problem as in Algorithm \ref{buildsets}, based on $\bar v_i'$ with $i=1,\hdots,5$,  we obtain  
	$\mathcal S_1'  =  \left\{1,2,5\right\}$, 
	$\mathcal S_2 ' =  \left\{1,2,4,5\right\}$, 
	$\mathcal S_3  '=  \left\{1,2,5\right\}$, 
	$\mathcal S_4  '=  \left\{1,2,3,5\right\}$ and 
	$\mathcal S_5  '=  \left\{1,2,5\right\}$,
and the universe of the minimum set covering problem is $\mathcal U=\{1,2,3,4,5\}$.
Finally, by inspection, we can see that a solution of this minimum set covering problem is the set of indices $\mathcal I'^*=\left\{2,4\right\}$. Hence, the sparsity of the  solution to the MCP coincides with  the solution to the MSCP associated with $\bar{A}$. Lastly, we observe that this example illustrates the conclusions of Proposition~\ref{zeroLebStruct} and Proposition~\ref{probstrucvsmin}.

\subsection{Robust Minimal Controllability Problem}\label{subSecrMCPexample}

Now, we illustrate how to find a solution to $\mathcal P_2$. Let us apply the developments of Section~\ref{rMCPsec}, when we consider the  dynamics' matrix in~\eqref{expa}.  First, if we consider that at most one input fails,  we use Algorithm~\ref{buildsets}, where a  set multi-covering problem  is considered with the sets  as in  Section~\ref{sec:MSCP},  universe $\mathcal U=\{1,\hdots,5\}$ and with a demand function $d(i)=2$ for $i=1,\ldots,5$, i.e., each element must be covered twice. Subsequently, by inspection, we conclude that the sets $\mathcal S_2$ and $\mathcal S_4$ need to be considered twice, since the elements $5$ and $4$ only appear in these sets, respectively. After this, we need to cover the element $2$ and to this end we can choose $\mathcal S_3$ or $\mathcal S_5$ or twice one of them, so a possible solution to the multi-set covering problem is $\mathcal M^\ast=\{2,3,4,2,3,4\}$. Therefore, $B_n(\mathcal M^\ast)$ is a solution to the rMCP, and, in particular,  the  solution is the same as concatenating twice a dedicated solution to the MCP, see Remark~\ref{CasoNotavel}. Further,  Algorithm~\ref{algorithm3} produces an optimal solution as often occurs in practice (Remark~\ref{remark6}).

In fact, if we apply the developments of Section~\ref{rMCPsec} when  $s$ inputs are allowed to fail, i.e., for demand function $d(i)=s+1$ for $i=1,\ldots,5$, we notice that the sets $\mathcal S_2$ and $\mathcal S_4$ need to be considered  $s+1$ times since the elements $5$ and $4$ only appear in these sets, respectively. Besides, we need to cover the element $2$, so we can choose either $\mathcal S_3$ or $\mathcal S_5$ $s+1$ times, which implies that $B(\mathcal M^\ast)$, with $\mathcal M^\ast=\{2,3,4,\hdots,2,3,4\}$ where the elements $2,3$ and $4$ appear $s+1$ times,  is a solution. Similarly, the solution consists of concatenating $s+1$ times a dedicated solution to the MCP, and the same remarks are applicable, i.e.,  Remark~\ref{CasoNotavel} and Remark~\ref{remark6}.

Notwithstanding, the concatenation of $s+1$ solutions to the MCP is not always a solution to the rMCP when at most $s$ inputs are allowed to fail. Let us consider the dynamics' matrix and associated left-eigenvectors as follows:
\begin{equation}\label{exp2a} A=\left[
\begin{array}{ccc}
 4 &\ -2 & \ 2 \\
 -1 &\ 3 & \ 1 \\
 1 & \ -1 & \ 5 \\
\end{array}\quad 
\right]\text{ and } \quad V=\left[
\begin{array}{ccc}
 | & | & | \\
 v^1 & v^2 & v^3 \\
 | & | & | \\
\end{array}
\right]=\left[
\begin{array}{ccc}
 1 & 0 & 1 \\
 1 & 1 & 0 \\
 0 & 1 & 1 \\
\end{array}
\right].\end{equation}

First, we note that $\sigma(A)=\{2,4,6\}$, so $A$ is simple, and we can apply the results in Section~\ref{rMCPsec}. Secondly, the structure of the left-eigenvectors of $A$ is given by $\bar v^1=[\begin{array}{ccc}\star&\star&0\end{array}]^\intercal$, $\bar v^2=[\begin{array}{ccc}0&\star&\star\end{array}]^\intercal$ and $\bar v^3=[\begin{array}{ccccc}\star&0&\star\end{array}]^\intercal$. Further,  we consider that  at most one input failure is likely to occur, i.e., $s=1$. Then, we can invoke Algorithm~\ref{buildsets} to build the sets  for the set multi-covering problem, which are as follows: $\mathcal S=\{\mathcal S_1,\mathcal S_2,\mathcal S_3\}$, with $\mathcal S_1=\{1,2\}$, $\mathcal S_2=\{2,3\}$ and $\mathcal S_3=\{1,3\}$, and $\mathcal U=\bigcup_{i=1}^3\mathcal S_i=\{1,2,3\}$. By inspection, we obtain that $\mathcal M'=\{1,2,3\}$ is the optimal solution, where the  indices cover each element of $\mathcal U$ twice. Further, observe that a solution to the dedicated input MCP always has  size equal to two, and in this case, the concatenation of two solutions lead to a solution that has one more input than the optimal solution obtained. Observe that this is a small dimensional example that incurs into a solution that is already $33\%$ worst than the optimal. Alternatively,  if we apply Algorithm~\ref{algorithm3} to approximate the solution to the rMCP, we obtain one that is optimal, i.e., $B(\mathcal M')$ where $\mathcal M'=\{1,2,3\}$, which is consistent with Remark~\ref{remark6}.

\section{Conclusions and Further Research}\label{conclusions}

In this paper, we addressed two minimal controllability problems, with the goal of characterizing the input configurations that actuate the  minimal subset of variables yielding controllability, under a specified number of failures. The problems explored were shown to be NP-complete, and a polynomial reduction of these to a set multi-covering problem was provided. In particular, the strategies followed by us separate the discrete and continues nature of the minimal controllability problems. Subsequently, we discussed  greedy solutions to the  minimal controllability problems that  yields feasible (but sub-optimal) solutions to rMCP. 

A possible, and interesting, direction for future research in this line of work includes the use of the obtained inputs' structure and consider methods such as coordinate gradient descent to minimize an energy cost, and to consider the case where the model is not exactly known.

%\small
\bibliographystyle{IEEEtran}
\bibliography{IEEEabrv,automatica13}

\appendix 

\paragraph*{Proof of Lemma~\ref{lemma1}}
Consider the sets $\mathcal S$ and $\mathcal U$ obtained in Algorithm~\ref{buildsets}.
The following equivalences hold: let $\mathcal{I}\subset\{1,\cdots,n\}$ be a set of indices and $\bar{b}_{\mathcal{I}}$ the structural vector whose $i$-th component is non-zero if and only if $i\in\mathcal{I}$. 
Then, the collection of sets $\{\mathcal{S}_{i}\}_{i\in\mathcal{I}}$ in $\mathcal S$ covers $\mathcal U$ if and only if
		$\forall j\in\mathcal J,~\exists k\in\mathcal{I}\text{ such that }\,j\in \mathcal S_k,$
		which is the same as
		$\forall j\in\mathcal J,~\exists k\in\mathcal{I} \text{ such that }\,\bar v^j_k\neq 0\text{ and }\,\bar b_k\neq 0\,$,
		this can be rewritten as 
		$\forall j\in\mathcal J,~\exists k\in\mathcal{I} \text{ such that } \bar v^j_k\bar b_k\neq 0$
		and therefore
		$\forall j\in\mathcal J\quad \bar v^j\cdot \bar b\neq 0.$
In summary, $\bar{b}_{\mathcal{I}}$ is a feasible solution to the problem in~\eqref{label3}. In addition, it can be seen  that by such reduction, the optimal solution $\bar{b}^\ast$ of~\eqref{label3} corresponds to the structural vector $\bar{b}_{\mathcal{I}^{\ast}}$, where $\{\mathcal{S}_{i}\}_{i\in\mathcal{I}^{\ast}}$ is the minimal collection of sets that cover $\mathcal{U}$, i.e., $\mathcal{I}^{\ast}$ solves the minimum set covering problem associated with $\mathcal{S}$ and $\mathcal{U}$.
 Hence, the result follows by observing that Algorithm~\ref{buildsets} has polynomial complexity, namely  $\mathcal O(\max\{|\mathcal J|,n\}^3)$.$\hfill\blacksquare$

\paragraph*{Proof of Theorem~\ref{acharb}} The proof follows by showing that if $\{v^i\}_{i\in\mathcal J}$  with countable $\mathcal J$ such that $v^{i}\neq 0$ for all $i\in\mathcal{J}$ and $\bar b$ a solution to~\eqref{label3}, then the set 
$
\Omega=\{b\in\mathbb{R}^n \ :  \ v^i \cdot b=0 \text{ for some } i\in\mathcal J, \text{ and } b \text{ is
a numerical instance of } \bar b \}
$
has zero Lebesgue measure. The proof  follows similar steps to those proposed in~\cite{Wonham}, but due to the additional sparsity constraint we devise an independent proof.  Let $\{v^{i}\}_{i\in\mathcal{J}}$, with countable $\mathcal{J}$, be given and let $\bar{b}$ be a solution to problem \eqref{label4}. For $b\in\mathbb{R}^{n}$,  the equation $v^i\cdot b=0$ represents a hyperplane $\mathcal{H}^i\subset\mathbb{C}^n$ (provided  $v^{i}\neq 0$ for all $i$), thus the equation $v^i\cdot b\neq0$ defines the space $\mathbb{C}^n\setminus \mathcal{H}^i$. Therefore, the set of $b$ that satisfies $v^i\cdot b\neq 0$ for all $i\in \mathcal J$, is given by
$\bigcap\limits_{i\in\mathcal J}\left (\mathbb{C}^n\setminus\mathcal{H}^i\right)=\mathbb{C}^n\setminus\left(\bigcup\limits_{i\in\mathcal J}\mathcal{H}^i\right)$
 and the set $\Omega$ of values which does not verify the equations is the complement, i.e.,
$\left(\mathbb{C}^n\setminus\bigcup\limits_{i\in\mathcal J}\mathcal{H}^i\right)^\mathsf{c}=\bigcup\limits_{i\in\mathcal J}\mathcal{H}^i,$
	which is a set with zero Lebesgue measure in $\mathbb{C}^n$, since $|\mathcal J|$ is countable.

Now, if $\{v^{i}\}_{i\in\mathcal{J}}$ is taken to be the set of left-eigenvectors of $A$ and $\bar{b}$ the corresponding solution to problem~\eqref{label4}, each member of the set $\Omega$ constitutes a solution to~\eqref{label4} and hence the MCP. Since, by the preceding arguments, $\Omega$ has Lebesgue measure zero in $\mathbb{C}^{n}$, it follows readily that almost all numerical instances of $\bar{b}$ are solutions to the MCP.$\hfill\blacksquare$

\paragraph*{Proof of Theorem~\ref{soundness}}
The existence of a solution is granted by Theorem~\ref{acharb}, and from~\cite{Megiddo84} one obtains the complexity for linear programming proposed.$\hfill\blacksquare$

\paragraph*{Proof of Lemma~\ref{lemma4}}
By Lemma \ref{lemma1}, given $\{\bar{v}^{i}\}_{i\in\mathcal{J}}$, problem \eqref{label4} is polynomially (in $|\mathcal{J}|$ and $n$) reducible to a  minimum set covering problem. Now, given a solution $\bar{b}$ to~\eqref{label3},   Algorithm~\ref{procedure} can be used to obtain a numerical instantiation  $b$ with the same structure as $\bar b$ such that $v^i\cdot b\neq 0$ for all $i\in\mathcal J$, which incurs polynomial complexity (in $|\mathcal{J}|$ and $n$) by Theorem~\ref{soundness}.  Furthermore, it is readily seen that any feasible solution $b'$ to~\eqref{label4} satisfies $\|b'\|_{0}\geq \|\bar{b}\|_{0}=\|b\|_{0}$. Hence, $b$ obtained by the above recipe is a solution to~\eqref{label4} and the desired assertion follows by observing that all steps in the above construction, yielding $\bar{b}$ have polynomial complexity (in $|\mathcal{J}|$ and $n$). $\hfill\blacksquare$

\paragraph*{Proof of Theorem~\ref{exactSolTheorem}} The proof follows by invoking the PBH eigenvector test in Theorem~\ref{popovvalvector} and the left-eigenbasis that is available by Assumption~1, and noticing that the problem in \eqref{label4} is a restatement of the MCP in  \eqref{mcpStatement}.$\hfill\blacksquare$

\paragraph*{Proof of Theorem~\ref{generalSol2MCP}}
The feasibility of the solution is ensured by proceeding similarly to Theorem~\ref{acharb} and Theorem~\ref{soundness}, when the left-eigenbasis of the dynamics' matrix is considered to invoke the PHB eigenvector criterion. The optimality follows similar steps to those presented in Lemma~\ref{lemma4}.$\hfill\blacksquare$

\paragraph*{Proof of Theorem~\ref{rMCPNP}}
The proof follows by noticing that we can polynomially reduce the MCP to an instance of the rMCP, and invoking Proposition~\ref{proposition2}.  In particular, the rMCP is already the result of such reduction since the MCP can be obtained  when the total number of inputs allowed to fail is $s=0$.$\hfill\blacksquare$

\paragraph*{Proof of Theorem~\ref{mainTheorMCP}}
First, we observe that, by construction of the sets $\{\mathcal S_1,\ldots, \mathcal S_{(s+1)n}\}$ and the demand function $d(i)$, for $i\in\{1,\ldots,n\}$, there exists always $s+1$ entries matching every non-zero entry of the vectors in a left-eigenbasis. This implies that if at most $s$ sensors fail, at least one entry of a column $c$ of $B$ is such that for each left-eigenvector $v.c\neq 0$, implying ${v^{i}}^\intercal B\neq 0$ for $i\in\{1,\ldots,n\}$. Hence, the system is controllable by Theorem~\ref{popovvalvector}, and we have a feasible solution. Now we need to show that the solution is optimal, i.e., there is not another solution with less dedicated inputs to the rMCP. We will proceed by contradiction, so assume that there is a solution to a demand function $d(i)=w$ for $i\in\{1,\ldots,n\}$ and some $w<s+1$. Then, for some entry of a left-eigenvector $v$ it is only ensured the existence of $w$ columns in $B$ whose inner product is not zero. Therefore, if $w$ dedicated inputs fails, i.e., the corresponding columns of $B$ are now zero, then  $B$ is such that $v^\intercal B= 0$, for some eigenvector $v$. Thus, contradicting the assumption that there is a sparser solution to the rMCP.$\hfill\blacksquare$

\paragraph*{Proof of Theorem~\ref{generalSol2rMCP}}
 The proof follows similar steps to those presented in Theorem~\ref{generalSol2MCP}. In particular, one as to recall the notion of merging procedure and the guarantees obtained in Theorem~\ref{mainTheorMCP}.$\hfill\blacksquare$

\paragraph*{Proof of Theorem~\ref{thmain}}	 
From \cite{nphard}, we have that the MCP is NP-hard, and, in particular, the  minimum set covering problem can be polynomially reduced to it. Therefore,  we just need to show that the MCP assuming that $A$ comprises only simple eigenvalues and the left-eigenbasis is known, i.e., under the assumption made in this paper, can be reduced polynomially to the  minimum set covering problem.

To this end,  note that, given the set $\{\bar{v}^{i}\}_{i\in\mathcal{J}}$ of left-eigenvectors of $A$, the MCP is equivalent to problem~\eqref{label4}, the latter being polynomially (in $|\mathcal{J}|$ and $n$) reducible to the minimum  set covering 
problem (see Lemma~\ref{lemma4}). Since $|\mathcal{J}|=n$, the overall reduction to the minimum  set covering problem is polynomial in $n$ and the result follows by invoking Proposition~\ref{proposition1}.

Similar arguments hold with rMCP, where the problem was shown to be NP-hard in Theorem~\ref{NPhardrMCP}, and a reduction to the minimum set multi-covering problem can be obtained as in~Theorem~\ref{mainTheorMCP}.
$\hfill\blacksquare$

\paragraph*{Proof of Theorem~\ref{complexityAlg3}} First, notice  that the output of Algorithm~\ref{algorithm3}, i.e., $B_n(\mathcal M')$, is a feasible solution since the algorithm stops when each of the elements in the universe of the set multi-cover is $s+1$ times covered. 

The computational complexity of Algorithm~\ref{algorithm3} is obtained by the overall complexity of steps~1, 4 and~5. In step~1, we need to compute $(s+1)n$ sets, in step~5 at most  $n$ sets need to be considered, and, in step~4, $(s+1)$ iterations are performed, each with the number of steps of step~5, yielding $(s+1)n$ computational steps. Summing up the complexity of each step, Algorithm~\ref{algorithm3} has, in the worst case, computational complexity of order $\mathcal O(s n)$. In addition, notice that the performance attained in a multi-set covering problem is the same as in the rMCP, as a consequence of Theorem~\ref{thmain}. Furthermore, the solution obtained incurs in an optimality gap of at most $\mathcal O(\log{n})$ since the algorithm implements the greedy algorithm associated with submodular functions, as it is the case of the multi-set covering problem, and the result follows. 
$\hfill\blacksquare$

\end{document}